\documentclass[12pt]{article}
\usepackage{amsfonts}
\usepackage{amsmath}
\usepackage{amsmath}
\usepackage{verbatim}
\usepackage[active]{srcltx}
\include{srctex}
\allowdisplaybreaks[4]
\newtheorem{theorem}{Theorem}[section]

\newtheorem{example}[theorem]{Example}

\newtheorem{lemma}[theorem]{Lemma}

\let\Section=\section
\def\section{\setcounter{equation}{0}\Section}
\newenvironment{proof}[1][Proof]{\textbf{#1.} }{\ \rule{0.5em}{0.5em}}

\begin{document}

\title{A nonlinear
 stochastic heat equation: H\"older continuity and smoothness of the density
of the solution }

\author{Yaozhong Hu\thanks{Y.  Hu is
partially supported by a grant from the Simons Foundation
\#209206.}, \  David Nualart\thanks{ D. Nualart is supported by the
NSF grant DMS0904538. \newline
  }\, and Jian Song   \\
  }
\date{}
\maketitle

\begin{abstract}
In this paper, we establish a version of the Feynman-Kac formula for
 multidimensional stochastic heat equation driven by a general
semimartingale.  This  Feynman-Kac formula is  then  applied to
study some nonlinear stochastic heat equations driven by
nonhomogenous Gaussian noise:   First, it is obtained   an explicit
expression for the Malliavin derivatives of the solutions.   Based
on the representation we obtain the smooth property of the density
of the law of  the solution.  On the other hand,  we also obtain the
H\"older continuity of the solutions.
\end{abstract}

\setcounter{equation}{0}
\section{Introduction}
In this paper we  consider the following   nonlinear
stochastic heat equation:
\begin{equation}
\begin{cases}
\dfrac{\partial u}{\partial t}=\dfrac{1}{2}\triangle u
+b(u)+\sigma(u) \dot{W}(t,x),
\quad t\ge 0,\quad  x\in \mathbb{R}^d&\\
u(0,x)=u_0(x)\,, &
\end{cases}\label{e.1.1}
\end{equation}
where  $\Delta=\sum_{i=1}^d \frac{\partial ^2}{\partial x_i^2}$ is
the Laplace  operator, $b$ and $\sigma$ are globally Lipschitz
continuous functions, and $W$ is a zero mean Gaussian random field,
which is a Brownian motion in the time variable and it has a
nonhomogeneous spatial covariance with density $q(x,y)$ (see
(\ref{e1}) for the  precise  definition). Here $\dot{W}(t,x)$ denotes  the generalized  random field
$\displaystyle{\frac {\partial^{d+1} W}{\partial t \partial x_1 \cdots \partial x_d}}$.

The case of an homogeneous covariance kernel $q(x,y)=q(x-y)$ has
been studied  in  the seminal paper by Dalang \cite{dalang}.  In
this case, the existence, uniqueness and   H\"older continuity of $u(t,x)$ with respect to both
parameters $t$ and $x$ is obtained in \cite{sanz}  under
integrability conditions on the spectral measure $\mu$ of the noise. We extend these results to the nonhomogeneous case
in Section 4.

On the other hand, using the techniques of Malliavin calculus, and
assuming suitable nondegeneracy conditions, one can show that for a
fixed $(t,x)$, $t>0$, the random variable $u(t,x)$, solution to (\ref{e.1.1}),  has an absolutely
continuous probability law and the density is smooth. The results
that have been obtained so far along this direction can be
summarized as follows.

\begin{itemize}
\item[(i)]  In \cite{pardoux} Pardoux and Zhang considered Equation (\ref{e.1.1})
when $x$ is in the interval  $(0,1)$ with Dirichlet boundary
conditions, assuming that $W$ is a space-time white noise.
 In this case, if the coefficients are Lipschitz, then $u(t,x)$ has an absolutely continuous
  distribution for any $t>0$ provided $\sigma(u_0(x_0)) \not=0$ for some $x_0\in (0,1)$.
  The smoothness of the density in this framework was proved by Mueller
  and Nualart in \cite{mullernualart}, assuming that the coefficients are infinitely
  differentiable with bounded derivatives. On the other hand,  under the stronger nondegeneracy
  condition $|\sigma(x)| \ge c>0$, and smooth coefficients, Bally and Pardoux \cite{ballypardoux}
  proved that the law of any vector of the form $(u(t,x_1),
 \dots,
u(t,x_n))$,
$0\le x_1 < \cdots  < x_n \le 1$, $t>0$, has an infinitely
differentiable density, assuming Neumann boundary conditions on
$(0,1)$.
\item[(ii)]   For the $d$-dimensional heat equation with an homogeneous spatial covariance,  Nualart and Quer-Sardanyons have provided sufficient conditions for the  existence and  smoothness of the density  of $u(t,x)$ for $t>0$ and $x\in\mathbb{R}^d$,  assuming   $|\sigma(x)| \ge c>0$, in the paper
\cite{nualartquer} (see also \cite{dkmnx}).
\end{itemize}

An open problem for the stochastic heat equation with colored
spatial covariance is to derive the  existence and smoothness of the
density under a nondegeneracy condition of the form
$\sigma(u_0(x_0)) \not=0$ for some $x_0\in \mathbb{R}^d$. The main
purpose of this paper is obtain   new results in this direction. To
prove such   results we need to show that the norm of the Malliavin
derivative of the solution  $  \int_0^t \| D_s
u(t,x)\|^2_{\mathcal{H}} ds $ is either strictly positive almost
surely (for the absolute continuity) or it has negative moments of
all orders (for the smoothness of the density), where $\mathcal{H}$
is the Hilbert space associated with the spatial covariance.

We develop a new approach to prove these results based on the Feynman-Kac representation for the solution  to the heat equation with multiplicative noise driven by a general continuous semimartingale.  The main idea is to express  $ \| D_s u(t,x)\|_{\mathcal{H}}^2 $ as the norm in $L^2(  \mathbb{R}^d)$ of a function $V_{s,\xi}(t,x)$ given by $V_{s,\xi}(t,x)= \int_{\mathbb{R}^d} c(\xi, y) D_{s,y}u(t,x)  dy$, where $c$ is the square root   of the kernel $q$ as an operator. Then for any fixed $(s,\xi)$,  $V_{s,\xi}(t,x)$ satisfies the  linear stochastic heat equation with random coefficients
\begin{equation}  \label{g2}
\dfrac{\partial V_{s,\xi}}{\partial t}=\dfrac{1}{2}\triangle V_{s,\xi}+b'(u)V_{s,\xi}+\sigma'(u)
V_{s,\xi} \dot {W}(t,x)\,,
\quad  t\ge s, x\in \mathbb{R}^d,
\end{equation}
with initial condition $V_{s,\xi}(s,x)= c(\xi,x)\sigma(u(s,x))$.

In order to establish a Feynman-Kac representation for the solution to Equation (\ref{g2}) we need to assume that the covariance kernel $q(x,y)$ is non-singular and this implies the existence of a random field $W_1(t,x)$ such that  $\dot {W}(t,x)=\dfrac{\partial W_1} {\partial t} (t,x)$. Then, Equation (\ref{g2}) is a  particular case of a more general stochastic  heat equation of the form
\begin{equation}\label{heat}
  \dfrac{\partial V}{\partial t}(t,x) =\dfrac{1}{2} \triangle V(t,x)+ V   \dfrac{\partial F} {\partial t}  (t,x),
\end{equation}
where $\{F(t,x), t\ge 0, x\in \mathbb{R}^d\}$ is a  continuous
semimartingale in the sense of Kunita  \cite{kunita}, with  local
characteristic $b(t,x)= b'(u(t,x))$ and  $a(t,x,y)= \sigma'(u(t,x))
\sigma'(u(t,y)) q(x,y)$. In  Section 3 (see  Theorem  \ref{feynman})
we derive a   Feynman-Kac formula for the solution of (\ref{heat})
assuming that the functions $b$ and $a$ are bounded by $C(1+
|x|^\beta)$, and $C(1+ |x|^\beta+|y|^\beta)$ for some $0\le \beta
<2$. This result has its own interest.   The proof is based on a
generalized  It\^o formula proved in \cite{kunita}.

 There have been other papers on the  Feynman-Kac formula for the stochastic  heat equation.
  We can mention the recent works   \cite{hunuso} and
  \cite{hln} on the stochastic  heat
 equation  driven  by
fractional white noise.   We refer to the references  in these papers for related works.

In Section 4 we show the existence and uniqueness of a solution for
the general stochastic heat equation  (\ref{e.1.1}) with a
nonhomogeneous spatial covariance and we deduce the H\"older
continuity of the solution. This result is an extension of the
 results  proved in  \cite{sanz}. Finally, in Section 5, assuming that
 the covariance kernel is continuous and  under a nondegeneracy condition
 of the form $q(x_0,x_0)>0$ and $\sigma(u_0(x_0))\not =0$ for some $x_0\in \mathbb{R}$,
  we establish the absolute continuity of the law of the solution and the smoothness
  of the density  if the coefficients are
 smooth.
 
 To simplify the presentation we have assumed that the functions $b$ and $\sigma$ depend only on the variable $u$. All the results of this paper could be extended without difficulty to the case of coefficients $b(t,x,u)$ and $\sigma (t,x,u)$ such that they are Lipschitz and with linear growth in $u$, uniformly in $(t,x)\in [0,T] \times \mathbb{R}^d$ for any $T>0$. In this case, the nondegeneracy condition would be $\sigma(0, x_0, u_0(x_0)) \not =0$, for some $x_0 \in \mathbb{R}^d$. 

 The results of this paper can be extended to the stochastic heat equation on an open and bounded set $A\subset \mathbb{R}^d$, with Dirichlet boundary conditions. In this case,  the Feynman-Kac formula involves a $d$-dimensional Brownian motion  starting form a point $x\in A$, and killed when it leaves the set  $A$. On the other hand, the existence and smoothness of the density have been deduced, applying techniques of Malliavin calculus, for stochastic differential equations of the form $Lu= b(u) + \sigma(u) \dot{W}$, where $L$ is a differential operator more general than $\partial_t -\frac 12 \Delta$ (see, for instance, \cite{Ko,Ti} where $L$ is a pseudodifferential operator and \cite{nualartquer} where $L$  is a general parabolic or hyperbolic operator). In all these examples, one assumes   that $\sigma$ is bounded away from the origin. Our approach to handle a nondegeneracy of the form
 $\sigma(0,x_0, u_0(x_0))$ only works  if a Feynman-Kac representation  is available for  the corresponding  stochastic linear equation satisfied by the Malliavin derivative. This  happens, for instance,  for parabolic operators of the form $L=\partial_t-
 \sum_i b_i \partial_{x_i} -\frac 12 \sum_{i,j}a_{i,j} \partial^2_{x_i,x_j}$. The methodology developed in this paper could be extended to these operators,   replacing the  Brownian motion by the diffusion process with generator
$ \sum_i b_i \partial_{x_i} -\frac 12 \sum_{i,j}a_{i,j} \partial^2_{x_i,x_j}$.

\setcounter{equation}{0}

\section{Preliminaries}
\subsection{Malliavin calculus}
Let $(\Omega, \mathcal{F}, P)$ be a complete probability space.   Consider a family of zero mean   Gaussian random variables
$W=\{ W_t(\varphi), \varphi\in C_0^\infty(  \mathbb{R}^{d }), t\ge 0\}$, where
$C_0^\infty(   \mathbb{R}^{d })$ denotes  the space of infinitely differentiable functions
on $  \mathbb{R}^{d }$ with compact support, with covariance
\begin{equation}  \label{e1}
E\left[ W_t(\varphi)W_s(\psi)\right] = ( t\wedge s) \int_{\mathbb{R}^{2d}} \varphi( x)\psi( y) q(x, y)dxdy,
\end{equation}
where $q$ is a   nonnegative definite   and locally integrable function.

 Let  $\mathcal H$ be the Hilbert space defined as the completion of $C_0^\infty(  \mathbb{R}^{d })$ by the  inner product
\[
\langle \varphi, \psi \rangle_\mathcal{H} :=  \int_{\mathbb{R}^{2d}} \varphi( x)\psi( y) q(x, y)dxdy.
\]
 The mapping $ \mathbf{1}_{[0,t]}\varphi \mapsto W_t(\varphi)$  can be extended to a linear isometry between $\mathcal{H}_\infty:=L^2([0,\infty);\mathcal H)$ and the $L^2$ space spanned by $W$. Then $\{W(h),  h\in \mathcal{H}_\infty\}$ is an isonormal Gaussian process associated with the Hilbert space  $\mathcal {H}_\infty$.

We will denote by $D$ the derivative operator in the sense of
Malliavin calculus. That is, if $F$ is a smooth and cylindrical
random variable of the
form%
\begin{equation*}
F=f(W(h _{1}),\ldots ,W(h_{n})),
\end{equation*}%
$h_{i}\in \mathcal{H}_\infty$, $f\in C_{p}^{\infty }(\mathbb{R}^{n})$
($f$ and
all its partial derivatives have polynomial growth), then $DF$ is the $%
\mathcal{H}_\infty$-valued random variable defined by
\begin{equation*}
DF=\sum_{j=1}^{n}\frac{\partial f}{\partial x_{j}}(W(h
_{1}),\ldots ,W(h_{n}))h _{j}.
\end{equation*}%
The operator $D$ is closable from  $L^{2}(\Omega )$ into $L^{2}(\Omega ;%
\mathcal{H}_\infty)$ and we define the Sobolev space $\mathbb{D}^{1,2}$
as the closure of the space of smooth and cylindrical random
variables under the
norm%
\begin{equation*}
\left\| DF\right\| _{1,2}=\sqrt{E(F^{2})+E(\left\| DF\right\| _{\mathcal{H}_\infty%
}^{2})}.
\end{equation*}%
We denote by $\delta $ the adjoint of the derivative operator,
given by duality formula
\begin{equation}
E(\delta (u)F)=E\left( \left\langle DF,u\right\rangle
_{\mathcal{H}_\infty}\right) , \label{dua}
\end{equation}%
for any $F\in \mathbb{D}^{1,2}$ and any element $u\in $ $L^{2}(\Omega ;%
\mathcal{H}_\infty)$ in the domain of $\delta $.  The operator $\delta $
is also called the Skorohod integral.  The higher Malliavin
derivatives can be defined in similar way  and we can define
$  \mathbb{D}^{k,p}$ for any integer $k\ge1$ and real number $p\ge1$. Set $\mathbb D^\infty=
\displaystyle \bigcap_{k\ge 1\,, p\ge 2}  \mathbb{D}^{k,p}$.
To obtain the existence and smoothness of the density,  we  make use of the following criteria.

\begin{theorem}  Let $F:\Omega\rightarrow \mathbb{R}$ be a random variable.
If $\displaystyle F\in \mathbb{D}^{1,2}$ and  $\|DF\|_{\mathcal{H}_\infty}>0$ almost surely, then the
probability law of $F$ is absolutely continuous with respect to the Lebesgue measure. Moreover, if $\displaystyle F\in \mathbb{D}^\infty$ and $E\left[ \|DF\|_{\mathcal{H}_\infty}^{-p}\right]<\infty$ for all $p\ge 1$,  then the   density of $F$
is infinitely differentiable.
\end{theorem}

For the proof of this result   and  a detailed presentation   of the Malliavin calculus we refer to
\cite{nualart} and the references therein.

\subsection{Generalized It\^o formula}
In this section we introduce some preliminaries on continuous
semimartingales depending on a parameter and the corresponding
generalized It\^o formula. We refer to \cite{kunita} for more
details.

Fix a time interval $[0,T]$,  a complete probability space $(\Omega,
\mathcal F, P)$ and a filtration $\{\mathcal F_t, 0\le t\le T\}$
satisfying the usual conditions (increasing, right-continuous, and
$\mathcal{F}_0$ contains  all the null sets). Let $\{F(t,x), 0\le
t\le T, x\in O\}$ be a family of real valued processes with
parameter $x\in O$,  where $O$ is a domain in $\mathbb R^d$. We can
regard it as random field with double parameters $x$ and $t$. If
$F(t,x )$ is $m$-times continuously differentiable with respect to
$x$ a.s. for any $t$, it can be regarded as stochastic process with
values in $C^m$ or a $C^m$-process. Here we denote by
$C^m=C^m(O,\mathbb{R})$ the set of all  real valued functions on $O$
which are $m$  times continuously differentiable. If furthermore,
for each multi-index $\alpha\in \{1,\dots, d \}^k$ with
$|\alpha|=k\le m$, $\{D^\alpha_x F(t,x),x\in O\}$ is a family of
continuous semimartingales, then $F(t,x)$ is called a
$C^m$-semimartingale.  Here we  have used  the notation $D^\alpha_x=
\frac {\partial^{|\alpha|}}{ \partial x_{\alpha_1} \cdots  \partial
x_{\alpha_k}}$.

We denote by  $C^{1,1}$    the set of all functions on
$a:[0,T]\times O\times O\rightarrow \mathbb{R}$  such that the
partial derivatives $\frac {\partial a}{\partial x_i}(t,x,y)$,
$\frac{\partial a}{\partial y_j}  (t,x,y)$ and $ \frac{\partial^2
a}{\partial x_i \partial y_j}  (t,x,y)$ exist for any $1\le i,j  \le
d$ and are continuous  in $(x,y)$, and  for any compact set
$K\subset O$     and  $1\le i,j  \le d$
 \[
 \int_0^T  \sup_{x,y\in K}  \left( |a(t,x,y)| +
  |\frac{\partial a}{\partial x_i}(t,x,y)|+ |\frac{\partial a}{\partial y_j}
   (t,x,y)|+ |\frac{\partial^2 a}{\partial x_i \partial y_j} (t,x,y)| \right) dt
 <\infty.
 \]
We also denote  $C^{1}$   the set of all functions on $b:[0,T]\times
O \rightarrow \mathbb{R}$   which are continuously differentiable in
$x$, and  for any compact set $K\subset O$    and any $1\le i  \le
d$
  \[
 \int_0^T  \sup_{x \in K}  \left( |b(t,x)| + |\frac{\partial b}{\partial x_i} (t,x) |\right) dt<\infty.
 \]

Let $ \{F(t,x),x\in O\}$  be a family of continuous semimartingales decomposed as $F(t,x)=M(t,x)+B(t,x)$, where $M(t,x)$ is a continuous local martingale and $B(t,x)$ is a continuous process of bounded variation. Let $A(t,x,y)$ be the joint quadratic variation of $M(t,x)$ and $M(t,y)$ and assume that $A(t,x,y)=\int_0^t a(s,x,y)ds$ and $B(t,x)=\int_0^t b(s,x)ds$, where $a(t,x,y)$ and $b(t,x)$ are predictable processes. Then $(a(t,x,y), b(t,x))$ is called the local characteristic of the family of semimartingales $\{F(t,x), x\in O\}$. Following Section 3.2 of  \cite{kunita},  we say that
the local characteristic $(a,b)$ belongs to the class $B^{1,0}$ if $a(t,x,y)$ and $b(t,x)$ are predictable processes
 with values in $C^{1,1}$ and $C^1$, respectively.

Now let $\{F(t,x),x\in O\}$ be a continuous semimartingale with local characteristic $(a,b)$.  Let $\{f_t, 0\le t\le T\}$ be a predictable process with values in $O$ satisfying
\begin{equation} \label{g1}
\int_0^Ta(s,f_s,f_s)ds<\infty,\quad \int_0^T|b(s,f_s)|ds<\infty\quad a.s.
\end{equation}
Then, the It\^o stochastic integral of $f_t$ based on the kernel $F(dt,x)$ is defined as the following limit in probability if it exists
\[\int_0^t F(ds,f_s)=\lim_{|\Delta|\to 0} \sum_{k=0}^{n-1}\{F(t_{k+1}\wedge t, f_{t_{k}\wedge t})-F(t_k\wedge t, f_{t_k\wedge t})\},\]
  where $\Delta=\{0=t_0<\cdots<t_n=T\}$, and $|\Delta| =\max_{1\le i \le n} (t_{i}- t_{i-1})$.

The joint quadratic variation of the It\^o integrals $\int_0^t F(ds, f_s)$ and $\int_0^t F(ds,g_s)$ satisfies
\[\langle \int_0^\cdot F(ds,f_s),\int_0^\cdot  F(ds,g_s) \rangle_t=\int_0^t a(s,f_s,g_s)ds.\]

The following  is the generalized It\^o formula (see Theorem 3.3.1 in \cite{kunita}).
\begin{theorem}[Generalized It\^o formula]\label{ito}
Let $\{F(t,x), x\in O\}$ be a continuous $C^2$-process and a continuous $C^1$-semimartingale with local characteristic belonging to the class $B^{1,0}$ and let $\{X_t, 0\le t\le T\}$ be a continuous semimartingale with values in $O$. Then $\{F(t,X_t), 0\le t\le T\}$ is a continuous semimartingale and satisfies
 \begin{align}
F(t,X_t)=&F(0,X_0)+\int_0^t F(dr,X_r)+\sum_{i=1}^d\int_0^t\dfrac{\partial F}{\partial x_i}(r,X_r) dX_r^i
\notag \\
&+\dfrac{1}{2} \sum_{i,j=1}^{d} \int_0^t
\dfrac{\partial^2F}{\partial x_i\partial x_j}(r,X_r)d\langle
X^i,X^j\rangle_r+\sum_{i=1}^d \langle\int_0^\cdot \dfrac{\partial
F}{\partial x_i}(dr,X_r), X^i\rangle_t\,,\label{e.2.2}
\end{align}
for any $t\in [0,T]$.
\end{theorem}

\setcounter{equation}{0}
\section{Feynman-Kac formula}
In this section we establish a general Feynman-Kac formula for the $d$-dimensional heat equation driven by a continuous semimartingale.
Suppose that $F=\{F(t,x), 0\le t\le T, x\in \mathbb{R}^d\}$ is a continuous  semimartingale with local characteristic $(a,b)$. We are going to impose the following condition.

\medskip
\noindent
\textbf{(H1)}  Assume that  $a(t,x,y)$ and $b(t,x)$ are  continuous   and satisfy
\begin{eqnarray}
&&|a(t,x,y)|\le C(1+|x|^\beta+|y|^\beta )  \label{eq7},
 \\
&& | b(t,x)| \le  C(1+|x|^\beta), \label{eq8}
\end{eqnarray}
 for $t\in[0,T]$,
with $0\le \beta <2$.

\medskip
 Consider the stochastic  heat equation
\begin{equation}  \label{eq1}
\left\{
\begin{array}{lc}
  \dfrac{\partial V}{\partial t}(t,x) =\dfrac{1}{2} \triangle V(t,x)+ V  \dfrac {\partial F}{\partial t}(t,x)  &   \\
   V(x,0)=h(x). &   \\
\end{array}
\right.
\end{equation}
An adapted random field $\{V(t,x),  0\le t\le T, x\in \mathbb{R}^d\}$ is called a mild solution to the above equation if $V(t,x)$ satisfies the following integral equation
\begin{equation}  \label{eq11}
V(t,x)=\int_{\mathbb{R}^d}p_t(x-z)h(z)dz+  \int_{\mathbb{R}^d} \left(\int_0^t
p_{t-r}(x-z)V(r,z) F(dr,z) \right) dz ,
\end{equation}
where $p_t(x)= (2\pi t) ^{-\frac d2} \exp(-|x|^2/2t)$.

\begin{theorem}[Feynman-Kac Formula]\label{feynman}
  Let   $h(x) $ be continuous
and with polynomial growth. Then  the process
\begin{equation}
 V(t,x)=E^B\left(h(x+B_t) \exp\left(  \int_0^t F(dr,
x+B_t-B_r)-\frac{1}{2}\int_0^t \bar a(r, x+B_t-B_r)dr  \right)\right),
\label{e.3.2}
\end{equation}
where $B$ is a $d$-dimensional standard Brownian motion
independent of $F$, $E^B$ denotes the mathematical expectation with respect to $B$, and $\bar a(t, x)=a(t, x, x)$, is a mild solution to Equation (\ref{eq1}).
\end{theorem}

\noindent\begin{proof}

We  divide the proof into three steps.

\medskip
\noindent
\textbf{Step 1}.\quad
First we show that the process  (\ref{e.3.2}) is well defined.   In the sequel we denote by $E$  the mathematical expectation in the probability space where $F$ is defined, and   $E^B$ denotes the expectation with respect to the independent Brownian motion $B$. Set
\[
Y_t= \int_0^tF(dr, x+B_t-B_r)-\dfrac{1}{2}\int_0^t\bar a(r,x+B_t-B_r)dr.
\]
Notice that  $ \int_0^tF(dr, x+B_t-B_r)$ is a well defined It\^o stochastic integral, because the process $\{B_t-B_r, 0\le r\le t\}$  is independent of the  semimartingale  $F$, and conditions (\ref{g1}) are satisfied.  Then $V(t,x)= E^B\left(h(x+B_t) \exp(Y_t)\right)$. We claim that this expectation exists and $V(t,x)$ satisfies the following condition for any   $x\in \mathbb{R}^d$ and $p\ge 1$,
\begin{equation}
\label{eq14}
\sup_{0\le t\le T} E |V(t,x) |^p  \le K_1 \exp\left( K_2 |x|^\beta \right),
\end{equation}
where the constants $K_1$ and $K_2$ depend on $p$ and $T$. In particular,    this implies that the stochastic integral in (\ref{eq11}) is well defined.  We can write
\[
E|V(t,x)|^p \le \left( E^B |h(x+B_t|^{2p}\
 E  E^B\exp(2pY_t)  \right)  ^{\frac 12}.
\]
Let us denote by $M(t,x)$ the martingale part of $F(t,x)$. Then we make the decomposition
\[
Y_t= Y_t^{(1)} +Y_t^{(2)} ,
\]
where
\[
Y_t^{(1)} =\int_0^tM(dr, x+B_t-B_r)-p\int_0^t \bar a(r,x+B_t-B_r) dr,
\]
and
\[
Y_t^{(2)} =\int_0^t \left[b(r, x+B_t-B_r)
+\left( p-\frac 12 \right)\bar a(r,x+B_t-B_r)  \right]dr.
\]
Using conditions (\ref{eq7}) and (\ref{eq8}) and taking into account that $\beta <2$, we obtain for all $t\in [0,T]$
\begin{eqnarray*}
E^B\exp(2pY_t^{(2)} )& \le& E^B \left( \exp\left( C  \int_0^t (1+|x+B_t-B_r|^\beta) dr  \right)\right)  \\
& \le& K_1 \exp(K_2 |x|^\beta).
\end{eqnarray*}
On the other hand, taking into account that $\exp(2pY_t^{(1)}) $ is a martingale, we can write
\[
E E^B\exp(2pY_t^{(1)} )=E^B E \exp(2p Y_t^{(1)}) \le 1,
\]
which  completes the proof of (\ref{eq14}).

\medskip
\noindent \textbf{Step 2}.\quad We  now   show    that the process
(\ref{e.3.2}) is a solution to Equation (\ref{eq1}) under some
additional regularity assumptions on the semimartingale $F(t,x)$.
Suppose that  $F(t,x)$ is  a $C^3$-semimartingale, such that the
local characteristic $(a,b)$ satisfies (\ref{eq7}) and (\ref{eq8}).
We also assume that  the functions $D^\alpha_xD^\alpha_y a(t,x,y)$
satisfy the estimate (\ref{eq7}) for all multi-index $\alpha$ with
$1\le |\alpha| \le 2$, and the functions  $D^\alpha_x  b(t,x )$  and
$D^\alpha_x  \bar a(t,x )$ satisfy the estimate (\ref{eq8}) for all
multi-index $\alpha$ with $1\le |\alpha| \le 2$. Clearly this implies
that   the  local characteristic belongs to the class $B^{1,0}$.
Suppose also that the functions  $D^\alpha_x  h(t,x )$     have
polynomial growth for all multi-index $\alpha$ with $1\le |\alpha|
\le 2$.

For fixed $x$, let
\[
\Phi(t,y)=\int_0^tF(dr, x+y-B_r)-\dfrac{1}{2}\int_0^t\bar a(r,x+y-B_r)dr.
\]
  According to Theorem 3.3.3 in \cite{kunita},  $\Phi(t,y)$ is a $C^2$-semimartingale with  local characteristic belonging to the class $B^{1,0}$.
  We can apply the generalized It\^o formula  (\ref{e.2.2})  to the process $Y_t=\Phi(t,B_t)$, and we obtain
\begin{align*}
dY_t
&=  F(dt,x)-\dfrac{1}{2}\bar a(t,x) dt+  \sum_{i=1}^d \left[ \int_0^t
\dfrac{\partial F}{\partial x_i} (dr, x+B_t-B_r)\right] dB_t^i\\
&-\dfrac{1}{2}  \sum_{i=1}^d \left[ \int_0^t \dfrac{\partial \bar a  }{\partial
x_i} (r, x+B_t-B_r) dr \right]   dB_t^i+\dfrac{1}{2}  \sum_{i=1}^d\left[  \int_0^t
\dfrac{\partial^2 F}{\partial x_i^2} (dr, x+B_t-B_r)\right]
dt\\
&-\dfrac{1}{4}   \sum_{i=1}^d \left[\int_0^t \dfrac{\partial^2 \bar a}{\partial
x_i^2} (r, x+B_t-B_r) dr\right]  dt.
\end{align*}
The terms $\langle \int_0^\cdot \frac{\partial F} {\partial
x_i}(dr,x),B_\cdot^i\rangle_t =\langle \int_0^\cdot \frac{\partial
M}{\partial x_i} (dr,x), B_\cdot^i\rangle_t$ vanish  since $M$ and
$B$ are independent.   The quadratic variation of the semimartingale
$Y$ is given by
\begin{equation} \label{eq3}
d\langle Y\rangle_t=
\bar a(t,x) dt +\sum_{i=1}^d\big[\int_0^t \dfrac{\partial F}{\partial
x_i}(dr, x+B_t-B_r)-\dfrac{1}{2}\int_0^t \dfrac{\partial
\bar a}{\partial x_i} (r, x+B_t-B_r)dr \big]^2dt \,.
\end{equation}

Consider the process  $Z(t,x)=h(x+B_t)e^{Y_t}$.
Applying It\^o's formula to $h(x+B_t)e^{Y_t}$  yields
\begin{align}   \notag
Z(t,x)
=&h(x)+  \int_0^t Z(s,x)dY_s+ \sum_{i=1}^d
\int_0^t \dfrac{\partial  h}{\partial x_i }(x+B_s)e^{Y_s}dB^i_s \\ \notag
&+   \dfrac{1}{2}
\sum_{i=1}^d
\int_0^t \dfrac{\partial^2 h}{\partial x_i^2}(x+B_s)e^{Y_s}ds
+ \dfrac{1}{2} \int_0^t  \hat V(s,x) d\langle
Y\rangle_s  \\ \label{eq15}
&+\sum_{i=1}^d \int_0^t \dfrac{\partial h}{\partial
x_i}(x+B_s)e^{Y_s}d\langle B^i, Y\rangle_s \,.
\end{align}
We claim that  the stochastic integrals with respect to $B^i$  in the above expression have zero expectation with respect to $B$.   This is a consequence of  the following properties
\begin{equation}  \label{eq4}
 \int_0^T   E^B Z(t,x)^2
E^B \left| \sum_{i=1}^d \int_0^t \dfrac{\partial F}{\partial
x_i}(dr, x+B_t-B_r)\right| ^2 dt <\infty,
\end{equation}
\begin{equation}  \label{eq5}
 \int_0^T  E^B Z(t,x)^2
E^B \left| \sum_{i=1}^d
 \int_0^t \dfrac{\partial
\bar a}{\partial x_i} (r, x+B_t-B_r)dr \right| ^2 dt <\infty,
\end{equation}
and
\begin{equation}
  \int_0^T E^B \left|\sum_{i=1}^d
\int_0^t \dfrac{\partial  h}{\partial x_i }(x+B_s) \right|^2e^{2Y_s}ds <\infty.
\end{equation}
These properties follow from our additional assumptions. For instance, to show (\ref{eq4}) for the martingale component of $F$, we take the expectation  in the probability space where $F$ is defined and  we use the fact that for any $p\ge 2$
\begin{eqnarray*}
&& E\left| \sum_{i=1}^d \int_0^t \dfrac{\partial M}{\partial
x_i}(dr, x+B_t-B_r)\right| ^p  \\
&& \qquad \le c_p E \left| \sum_{i,j=1}^d\int_0^t\dfrac{\partial^2 a}{\partial
x_i\partial y_j}(r,x+B_t-B_r, x+B_t-B_r)dr \right| ^{\frac p2} \\
&& \qquad \le  C   E\int_0^t  (1+ |B_t-B_r| ^\beta) ds <\infty.
\end{eqnarray*}
Then, taking the expectation with respect to $B$ in  (\ref{eq15}) yields
\begin{align*}
&V(t,x)
= h(x)+\int_0^tV(s,x)F(ds,x)\\
& + \dfrac{1}{2} \sum_{i=1}^d E^B \bigg(\int_0^t   V(s,x)
\bigg\{\int_0^s \dfrac{\partial^2 F}{\partial
x_i^2}(dr,x+B_s-B_r)-\dfrac{1}{2}\int_0^s \dfrac{\partial^2
\bar  a}{\partial
x_i^2}(r,x+B_s-B_r) dr \\
&+\bigg[\int_0^s\dfrac{\partial F}{\partial
x_i}(dr,x+B_s-B_r)-\dfrac{1}{2}\int_0^s \dfrac{\partial
\bar  a}{\partial x_i}(r,x+B_s-B_r) dr\bigg]^2\bigg\}ds\\
&+ \int_0^t \dfrac{\partial^2 h}{\partial
x_i^2}(x+B_s)e^{Y_s}ds\\
&+2\int_0^t\dfrac{\partial h}{\partial x_i}(x+B_s)
e^{Y_s}[\int_0^s \dfrac{\partial F}{\partial
x_i}(dr,x+B_s-B_r)-\dfrac{1}{2}\int_0^s \dfrac{\partial
\bar  a}{\partial x_i} (r,x+B_s-B_r) dr ]ds\bigg)
\end{align*}
Using  that
\[
\dfrac{\partial Y_s}{\partial x_i}=\int_0^s
\dfrac{\partial F}{\partial
x_i}(dr,x+B_s-B_r)-\dfrac{1}{2}\int_0^s \dfrac{\partial
\bar  a}{\partial x_i}(r,x+B_s-B_r)dr\,,
\]
we  obtain easily
\[
V(t,x)= h(x)+\int_0^t V(s,x)F(ds,x)+\dfrac{1}{2}
\sum_{i=1}^d \int_0^t \dfrac{\partial^2 V}{\partial
x_i^2}(s,x)ds\,.
\]
This shows that under some the additional regularity conditions on $F$
and $h$  the process
$u$ defined by  (\ref{e.3.2}) is a strong solution to Equation
(\ref{eq1}), and also a mild solution.

\medskip
\noindent
\textbf{Step 3}.\quad  Consider now the case of a  general semimartingale $F$. For any $\epsilon>0$ we define
\begin{eqnarray*}
M^\epsilon(t,x)&=&\int_{\mathbb{R}^d}M(t,y)p_\epsilon(x-y)dy, \\
B^\epsilon(t,x)&=&\int_{\mathbb{R}^d}B(t,y)p_\epsilon(x-y)dy,
\end{eqnarray*}
and
 $h^{\epsilon}(x)=\int_{\mathbb{R}^d}h(y)p_\epsilon(x-y)dy$.
It is easy to  check that $h^\epsilon$  is infinitely differentiable
and it has polynomial growth together with all its partial
derivatives.    Also $F
^\epsilon(t,x)=M^\epsilon(t,x)+B^\epsilon(t,x)$ is a
$C^3$-semimartingale with local characteristic given by
\[
b^\epsilon(t,x)=\int_{\mathbb{R}^d}b(t,y)p_\epsilon(x-y)dy,
\]
and
\[
a^\epsilon(t,x,y)=\int_0^t\int_{\mathbb{R}^{2d}} a(s,x-z_1,y-z_2)
p_\epsilon(z_1)p_\epsilon(z_2)dz_1dz_2ds.
\]
It easy to check that $a^\epsilon$ and $b^\epsilon$ satisfy the
estimates (\ref{eq7}) and (\ref{eq8}) respectively,  the partial
derivatives $D^\alpha_xD^\alpha_y a^\epsilon(t,x,y)$ satisfy the
estimate (\ref{eq7}) for all multi-index $\alpha$ with $1\le |\alpha|
\le 2$, and the functions  $D^\alpha_x  b^\epsilon(t,x )$    satisfy
the estimate (\ref{eq8}) for all multi-index $\alpha$ with $1\le
|\alpha| \le 2$.   From Step 2 it follows that
\[
V^\epsilon(t,x)=E^B\big\{h^\epsilon(x+B_t)\exp \big(\int_0^t F^\epsilon(dr,
x+B_t-B_r)-\dfrac{1}{2}\int_0^t \bar a^\epsilon(dr,
x+B_t-B_r)\big)\big\}
\]
is the strong solution to
\begin{equation} \label{heat1}
\left\{
\begin{array}{lc}
  \dfrac{\partial V^\epsilon}{\partial t}(t,x) =\dfrac{1}{2}
  \triangle V^\epsilon(t,x)+ V^\epsilon  \dfrac  {\partial F^\epsilon} {\partial t} (t,x)  &   \\
   V^\epsilon(x,0)=h^\epsilon(x) &   \\
\end{array}
\right.
\end{equation}
As a consequence, it is also a mild solution to
(\ref{heat1}),   namely,
\[
V^\epsilon(t,x)=\int_{\mathbb{R}^d} p_t(x-z) h^\epsilon(z)dz+ \int_{\mathbb{R}^d}
 \left( \int_0^t  p_{t-r}(x-z)V^\epsilon(r,z)F^\epsilon(dr,z) \right)  dz   .
\]
Finally,  we are going to take the limit as $\epsilon$ tends to zero in each term of the above expression in order to deduce  the Feynman-Kac formula of
$V $.  The estimate   (\ref{eq7}) implies
\[
\sup\limits_{\epsilon>0}E\exp \left(p \int_0^t
|\bar  a^\epsilon(r,x+B_t-B_r)|  dr\right)<\infty,  .
\]
for all $p\ge 1$ and, as a consequence,
  $V^\epsilon(t,x)$ converges to $V(t,x)$ in $L^p$ for all $p\ge 1$. Clearly,
  \[
  \lim_{\epsilon \downarrow 0} \int_{\mathbb{R}^d} p_t(x-z) h^\epsilon(z)dz=\int_{\mathbb{R}^d} p_t(x-z) h(z)dz,
  \]
and
 \[
  \lim_{\epsilon \downarrow 0}\int_0^t \int_{\mathbb{R}^d}
p_{t-r}(x-z)V^\epsilon(r,z)     b^\epsilon(r,z)dzdr=\int_0^t \int_{\mathbb{R}^d}
p_{t-r}(x-z)V (r,z)     b (r,z)dzdr,
  \]
  also in $L^p$ for all $p\ge 1$.  The following limits in $L^p$ are also easy to check:
  \[
  \lim_{\epsilon \downarrow 0}  \int_{\mathbb{R}^d}\left( \int_0^t
p_{t-r}(x-z)V^\epsilon(r,z) [M^\epsilon(dr,z) -M( dr,z) ] \right)  dz  =0
\]
and
 \[
  \lim_{\epsilon \downarrow 0}  \int_{\mathbb{R}^d}\left( \int_0^t
p_{t-r}(x-z)[V^\epsilon(r,z)- V(r,z)]  M (dr,z)\right)  dz    =0.
\]
This completes the proof of the theorem.
\end{proof}

\setcounter{equation}{0}
\section{Stochastic heat equation: H\"older continuity of the solution}
Consider  the following  nonlinear stochastic partial differential equation:
\begin{equation}\label{heateq}
\begin{cases}
\dfrac{\partial u}{\partial t}=\dfrac{1}{2}\triangle u
+b(u)+\sigma(u) \dot {W}(t,x)\,,\quad  t\ge 0\,,\quad
x\in \mathbb{R}^d &\\
u(0,x)=u_0(x)\,. &
\end{cases}
\end{equation}
where $W$  is  the Gaussian family  introduced in Section 2.1 with covariance function given by (\ref{e1}).
Let us recall that an adapted   random field  $\{u(t, x)\,, t\ge
0\,, x\in \mathbb{R}^d \} $ is called a mild solution to Equation
(\ref{heateq})  if $u$ satisfies the following integral equation.
\begin{align}\label{heatint}
u(t,x)=&\int_{\mathbb{R}^d}
p_t(x-z)u_0(z)dz+\int_0^t\int_{\mathbb{R}^d}
p_{t-r}( x-z)b(u(r,z))dzdr\nonumber\\
  & +  \int_0^t\int_{\mathbb{R}^d}
p_{t-r}( x-z)\sigma(u(r,z))W(dr,dz) \,,
\end{align}
where the stochastic integral  is defined as the integral of an $\mathcal{H}$-valued predictable process.

We are going to impose the following condition on the covariance function.

\medskip \noindent
(\textbf{H1}) For each $t\ge 0$,
\[
\sup\limits_{x\in\mathbb{R}^d}\displaystyle
 \int_0^t\int_{\mathbb{R}^{2d}} p_{t-s}( x-z_1)p_{t-s}(
x-z_2)|q(z_1,z_2)|dz_1dz_2 ds <\infty\,.
\]

\begin{theorem}\label{existence}
Suppose that $b$ and $\sigma$ are globally Lipschitz continuous
functions and  suppose that  the covariance function $q$ satisfies  (\textbf{H1}).
Let  $u_0(x)$ be a bounded  function in $\mathbb{R}^d$.
Then there exists a unique adapted process $u=\{u(t,x), t\in[0,T],
x\in\mathbb{R}^d\}$ satisfying (\ref{heatint}).  Moreover,
\begin{equation}
\sup\limits_{t\in[0,T], x\in \mathbb{R}^d} E|u(t,x)|^p<\infty,\quad
\forall \ p\ge 2\,. \label{e.4.3}
\end{equation}
\end{theorem}
\def\BB{{\mathbb{B}}}
\begin{proof} Fix $p\ge 2$. Let $\BB_p$ be the Banach space of all adapted random
fields $u$ such that $\|u\|_p<\infty$, where
$\|u\|_p^p=\sup\limits_{t\in[0,T], x\in \mathbb{R}^d} E|u(t,x)|^p$.
On $\BB_p$, define  the following mapping
\begin{align*}
\Psi(u)(t,x):=&\int_{\mathbb R^d}
p_t(x-z)u_0(z)dz+\int_0^t\int_{\mathbb{R}^d}
p_{t-r}( x-z)b(u(r,z))dzdr\nonumber\\
  & +  \int_0^t\int_{\mathbb{R}^d}
p_{t-r}( x-z)\sigma(u(r,z))W(dr,dz) \,.
\end{align*}
It is straightforward to obtain
\begin{align*}
&E\left|\Psi(u)-\Psi(v)\right|^p(t,x)
\le   C\bigg[E\left(\int_0^t \int_{\mathbb{R}^d}
p_{t-s}( x-z)|u(s,z)-v(s,z)|dzds\right)^p\\
&+E\Bigg\{ \int_0^t\int_{\mathbb{R}^{2d}} p_{t-s}( x-z_1)|u(s,z_1)-v(s,z_1)|
p_{t-s}( x-z_2)|u(s,z_2)-v(s,z_2)|\\
& \quad  \times |q(z_1,z_2)| dz_1dz_2ds\Bigg\}^{p/2} \bigg]\,.
\end{align*}
Taking the supremum with respect to  $t$ and $x$, we have
\[
\|\Psi(u)-\Psi(v)\|_p^p\le C\int_0^T \|u-v\|_p^p  ds \le
CT\|u-v\|_p^p\,.
\]
 Consequently, $\Psi$ is a contraction mapping on
$\BB_p$ when $T$  sufficiently small.  This proves the existence and
uniqueness of the solution for some small $T$.  From the above
argument  it is clear that the $T$ such that $\Psi$ is a contraction
is independent of the initial value of the solution. This can be
used to to show the existence and uniqueness of the solution for any
$T$.  The inequality  (\ref{e.4.3}) follows  in a similar way.
\end{proof}

Now we apply the factorization method to  obtain  the H\"older
continuity of $u$.  Fix an arbitrary $\alpha\in (0, 1)$ and  denote
\begin{equation}\label{ya}
Y_\alpha(r,z)=\int_0^r\int_{\mathbb{R}^d} p_{r-s}(   z-y)\sigma(u(s,y))(r-s)^{-\alpha}W(ds,dy).
\end{equation}
The semigroup property of  the heat kernel and the stochastic
Fubini's theorem yield
\begin{align}\label{fubini}
&\int_0^t\int_{\mathbb{R}^d}
p_{t-s}( x-y)\sigma(u(s,y))W(ds,dy) \nonumber\\
=&\dfrac{\sin(\pi\alpha)}{\pi} \int_0^t\int_{\mathbb{R}^d}
p_{t-r}( x-z)(t-r)^{\alpha-1}Y_\alpha(r,z)dzdr.
\end{align}
Consider the following stronger condition on the covariance function.

\medskip \noindent
(\textbf{H1a}) There exists $\gamma>-1$ such that for each $t\ge 0$,
\[
\sup\limits_{x\in\mathbb{R}^d}\displaystyle \int_{\mathbb{R}^{2d}}
p_t(x-z_1)p_t(x-z_2) |q(z_1,z_2) |dz_1dz_2<Ct^{\gamma}.
\]

\begin{lemma}\label{bound}
Let the assumptions of  Theorem \ref{existence} be satisfied.  Assume
the covariance function $q$ satisfies  (\textbf{H1a}).
Then for any fixed $T>0, p\ge1, \alpha \in
(0,\dfrac{1+\gamma}{2}),$ we have
$$\sup\limits_{r\in [0,T], z\in\mathbb{R}^d}E(|Y_\alpha(r,z)|^p)<\infty.$$
\end{lemma}
\begin{proof}
Since $\sup\limits_{r\in [0,T],
z\in\mathbb{R}^d}E(|u(r,z)|^p)<\infty$ from Theorem
\ref{existence}, and $\sigma$ is Lipschitz continuous,  we have
$$\sup\limits_{r\in [0,T],
z\in\mathbb{R}^d}E(|\sigma(u(r,z))|^p)<\infty.$$
Then we can write
\begin{align*}
E|Y_\alpha(r,z)|^p
\le &
C\bigg(E\int_0^r\int_{\mathbb{R}^{2d}}p_{r-s}(   z-y_1)p_{r-s}(   z-y_2)\\
&\quad   \times \sigma(u(s,y_1))\sigma(u(s,y_2))(r-s)^{-2\alpha}q(y_1,y_2)dy_1dy_2ds
\bigg)^\frac{p}{2}\\
\le&C\sup\limits_{r\in [0,T],
z\in\mathbb{R}^d}E(|\sigma(u(r,z))|^p)\\
&\quad  \times \bigg(\int_0^r\int_{\mathbb{R}^{2d}}p_{r-s}(   z-y_1)p_{r-s}(
z-y_2)(r-s)^{-2\alpha}|q(y_1,y_2)|dy_1dy_2ds
\bigg)^\frac{p}{2}\\
\le &C \left(\int_0^r (r-s)^{\gamma-2\alpha}ds
\right)^\frac{p}{2}<\infty.
\end{align*}
\end{proof}

Equation (\ref{fubini})  and Lemma  \ref{bound}   constitute  the main ingredients
to prove the following theorem concerning the H\"older continuity of the solution $u$.
\begin{theorem}\label{holder}
Suppose that $b$ and $\sigma$ are globally Lipschitz continuous.
Assume (\textbf{H1a}) and suppose that
   $u_0(x)$ is bounded    and
$\rho$-H\"older continuous.
Then the   solution
$u$ to the equation (\ref{heateq})  is  a.s. $\beta_1$-H\"older continuous in the  time
variable $t$
and $\beta_2$-H\"older continuous in  the space variable $x$ for any $\beta_1 \in
(0,\dfrac{1}{2}[\rho\wedge(1+\gamma)])$ and $\beta_2 \in
(0,\rho\wedge(1+\gamma))$, respectively.
\end{theorem}
\begin{proof}
It suffices to follow the idea of the proof of  Theorem 2.1 in \cite{sanz}, and we provide a sketch of the proof for the reader's convenience.
 The proof contains two parts for time and space variables, respectively.\\
{\bf Part I}\\
Fix $T, h>0$ and $p\in[2,\infty).$ First we show that
\begin{equation}\label{I}
\sup_{0\le t\le T}\sup_{x\in\mathbb R^d} E(|u(t+h,x)-u(t,x)|^p)\le C(p,T)h^{\eta p},
\end{equation}
for any $\eta\in(0,\frac12[\rho\wedge(1+\gamma)])$.
Let $Y_\alpha$ be as defined in (\ref{ya}) with $\alpha\in(0,\frac{1+\gamma}{2})$ and denote
$P_tf(x)=\int_{\mathbb R^d} p_t(x-z)f(z)dz$.
We have
\[E(u(t+h,x)-u(t,x)|^p)\le C(p,\alpha) \sum_{i=1}^4 I_i(t,h,x),\]
where
\begin{align*}
I_1(t,h,x)&= |P_{t+h}u_0(x)-P_tu_0(x)|^p,\\
I_2(t,h,x)&=E\left(\left|\int_0^t  \int_{\mathbb R^d} [p_{t+h-r}(x-z)(t+h-r)^{\alpha-1}-p_{t-r}(x-z)(t-r)^{\alpha-1}]Y_\alpha(r,z)\right|^pdzdr\right),\\
I_3(t,h,x)&=E\left(\left|\int_t^{t+h}\int_{\mathbb R^d}  p_{t+h-r}(x-z)(t+h-r)^{\alpha-1}Y_\alpha(r,z)\right|^pdzdr\right),\\
I_4(t,h,x)&=E\left(\left|\int_0^{t+h}\int_{\mathbb R^d}  p_{t+h-r}(x-z)b(u(r,z))dzdr-\int_0^{t}\int_{\mathbb R^d}  p_{t-r}(x-z)b(u(r,z))\right|^pdzdr\right).
\end{align*}
For the term $I_1(t,h,x)$, using the fact that $u_0$ is $\rho$-H\"older continuous  we have
$I_1(t,h,x)\le Ch^{\frac{\rho p}2}$.
For any $\alpha\in(0,\frac{1+\gamma}{2})$ set $\psi^\alpha(t,x)=p_t(x)t^{\alpha-1}$.
By H\"older's inequality  and Lemma \ref{bound}, we have
\[I_2(t,h,x)\le C\left(\int_0^t\int_{\mathbb R^d} |\psi^\alpha(t+h-r, x-z)-\psi^\alpha(t-r,x-z)|dzdr\right)^p.\]
Set
\[I_{2,1}(t,h,x)=\int_0^t\int_{\mathbb R^d} \exp\left(-\frac{|x-z|^2}{2(t-r)}\right)|(t+h-r)^{\alpha-1-\frac d2}-(t-r)^{\alpha-1-\frac d2}|dzdr,\]
and
\[I_{2,2}(t,h,x)=\int_0^t\int_{\mathbb R^d}(t+h-r)^{\alpha-1-\frac d2}\left| \exp\left(-\frac{|x-z|^2}{2(t+h-r)}\right)-\exp\left(-\frac{|x-z|^2}{2(t-r)}\right)\right|dzdr.\]
Then $I_2(t,h,x)\le C(I_{2,1}(t,h,x)^p+I_{2,2}(t,h,x)^p)$, and using the same arguments as in the proof of  Theorem 2.1 in  \cite{sanz}  to estimate the terms $I_{2,1}(t,h,x)^p$ and $I_{2,3}(t,h,x)^p$, we obtain  for $\eta\in(0,\alpha)$ that  $I_2(t,h,x)\le C h^{\eta p}$. 
By H\"older's inequality and Lemma \ref{bound}, we have
\[I_3(t,h,x)\le C\left(\int_t^{t+h}\int_{\mathbb R^d} p_{t+h-r}(x-z)(t+h-r)^{\alpha-1}dzdr\right)^p\le Ch^{\alpha p}.\]
A change of variable yields
\[I_4(t,h,x)\le C (I_{4,1}(t,h,x)+I_{4,2}(t,h,x)),\]
with
\begin{align*}
  I_{4,1}(t,h,x)&=E\left(\left|\int_0^h\int_{\mathbb R^d} P_{t+h-r}(x-z)b(u(r,z))dzdr\right|^p\right),\\
  I_{4,2}(t,h,x)&=E\left(\left|\int_0^h\int_{\mathbb R^d} P_{t-r}(x-z)[b(u(r+h,z))-b(u(r,z))]dzdr\right|^p\right),\\
\end{align*}
Since $b$ is Lipschitz, using H\"older's inequality and Equation (\ref{e.4.3}), we have
\[I_{4,1}(t,h,x)\le Ch^p.\]
The Lipschitz property of $b$ also implies
\[I_{4,2}(t,h,x)\le \int_0^t \sup_{z\in \mathbb R^d}E(|u(r+h,z)-u(r,z)|^p)dr.\]
Putting together all estimations for $I_i, i=1,\dots,4$, we obtain
\[\sup_{x\in\mathbb R^d} E(|u(t+h,x)-u(t,x)|^p)\le C\left(h^{p\min(\frac\rho2,\eta,\alpha)}+\int_0^t \sup_{x\in\mathbb R^d} E(|u(r+h,x)-u(r,x)|^p)dr\right).\]
Since $0<\eta<\alpha<\frac{1+\gamma}{2},$ the estimate (\ref{I}) follows by Gronwall's Lemma.

\noindent
{\bf Part II}\\
Now consider the increments in the space variable. We want to show that for any $T>0, p\in[2,\infty), x,a\in\mathbb R^d$ and $\eta\in(0,\rho\wedge(1+\gamma)),$
\begin{equation}\label{II}
\sup_{0\le t\le T}\sup_{x\in\mathbb R^d} E(|u(t,x+a)-u(t,x)|^p)\le C|a|^{\eta p}.
\end{equation}
Fix $\alpha\in(0,\frac{1+\gamma}{2})$. Ee have
\[E(|u(t,x+a)-u(t,x)|^p)\le C\sum_{i=1}^3 J_i(t,x,a),\]
with
\begin{align*}
  J_1(t,x,a)&=|P_tu_0(x+a)-P_tu_0(x)|^p\\
  J_2(t,x,a)&=E\left(\left|\int_0^t\int_{\mathbb R^d} [\psi^\alpha(t-r,x+a-z)-\psi^\alpha(t-r,x-z)]Y_\alpha(r,z)dzdr\right|^p\right),\\
  J_3(t,x,a)&=E\left(\left|\int_0^t\int_{\mathbb R^d} [p_{t-r}(x+a-z)-p_{t-r}(x-z)]b(u(r,y))dzdr\right|^p\right).
\end{align*}
It is easy to show that $J_1(t,x,a)\le C|a|^{\rho p}$. 
For the term $J_2(t,x,a)$, first we have using the mean value theorem,
\[
\int_{\mathbb R^d}|\psi^\alpha(t-r, x+a-z)-\psi^\alpha(t-r,x-z)|dz \le C(t-r)^{\alpha-1-\frac \eta2}|a|^{\eta},
\]
where $\eta\in(0,1)$. Again by H\"older's inequality and Lemma \ref{bound}, for $\alpha\in(0,\frac{1+\gamma}{2}),\eta\in(0,2\alpha\wedge 1),$ we deduce
\[J_2(t,x,a)\le C\left(\int_0^t\int_{\mathbb R^d}|\psi^\alpha(t-r, x+a-z)-\psi^\alpha(t-r,x-z)|dzdr\right)^p\le C|a|^{\eta p}.\]
Finally, by a change of variable, the Lipschitz property of b, and H\"older's inequality,
\begin{eqnarray*}
 J_3(t,x,a)&  \le& E\left(\left|\int_0^t \int_{\mathbb R^d} p_{t-r}(x-y)[b(u(r,z+a)-b(u(r,z)]dzdr\right|^p\right)\\
 & \le&C \int_0^t \sup_{z\in\mathbb R^d} E(|u(r,a+z)-u(r,z)|^p)dr.
  \end{eqnarray*}
Then (\ref{II}) follows from the Gronwall's lemma and the estimates of $J_i,i=1,2,3$.
\end{proof}

 Here are two examples.

\begin{example}A similar  H\"older continuity result was obtained in
\cite{sanz} in the case of an homogeneous covariance function   $q(x,y)=q(x-y)\ge 0$, where $q$  is a nonnegative continuous function on $\mathbb{R}^d \backslash \{0\}$ such that it is the Fourier transform of a non-negative definite tempered measure $\mu$ on $\mathbb{R}^d$, and  for some $\eta\in (0,1)$ we have
\[
\displaystyle
\int_{\mathbb{R}^d}\dfrac{ \mu(d\xi)} {(1+|\xi|^2)^\eta} <\infty.
\]
This condition  implies (\textbf{H1a}).
In fact, we can write
\begin{eqnarray*}
&&\int_{\mathbb{R}^{2d}} p_t(x-z_1) p_ t(x-z_2) q(z_1-z_2) dz_1dz_2
=  \int_{\mathbb{R}^{2d}} p_t(x-z-z_2) p_{t}(x-z_2) q(z) dz_2dz\\
& =& \int_{\mathbb{R}^{d}} p_{2t} ( z)  q(z)  dz
=\int_{\mathbb{R}^{d}}   e^{-2t\xi^2}   \mu( d \xi) \\
& \le &
  \int_{\mathbb{R}^{d}}  e^{-2t(\xi^2+1)}    \mu(d\xi) \le Ct^{-\eta} \int_{\mathbb{R}^d}\dfrac{\mu(d\xi)}{(1+|\xi|^2)^\eta}.
\end{eqnarray*}
 \end{example}

 Theorem \ref{holder} can be applied to noises which do not have an homogeneous spatial covariance like the following example.

\begin{example}
Consider the case where $d=1$ and  the covariance structure in space is that of a bifractional Brownian motion
with parameters $H\in(0,1),K\in(0,1]$, that is,
\[
q(x,y)=2^{-K}\dfrac{\partial^2}{\partial x\partial y}((|x|^{2H} + |y|^{2H})^K - |x-y|^{2HK}),
\]
 where $2KH>1$.
Then, $ B^{H,K}(t,x)= W(\mathbf{1}_{[0,t]\times[0,x]})$, is a bifractional Brownian motion in $x\ge 0$ for each fixed $t$, and formally,
$W(t,x)=\frac{\partial }{\partial x} B^{H,K}( t,x)$.     Then
\[
|q(x,y)|
\le  C[|x|^{2HK-2}+|y|^{2HK-2}+|x-y|^{2HK-2}]
\]
and $\gamma=HK-1\in(-1,0).$  Thus, Theorem \ref{holder}
can be applied to this case.
\end{example}

\setcounter{equation}{0}
\section{Stochastic heat equation: Regularity of the density of the solution}
In this section we consider again the solution
 $u(t,x)$ to
 (\ref{heateq}), and we will impose  the following  condition  on the covariance $q(x,y)$.

 \medskip \noindent
 (\textbf{H2}) $q $ is $\gamma_0$-H\"older continuous for some $\gamma_0>0$, and for  some $\beta \in[0,2)$
 \[
 |q(x_1,x_2)|\le C(1+|x_1|^\beta+|x_2|^\beta).
 \]
In this case we can assume that  the random field  $\{W_t(\varphi)\}$ has a density with respect to the Lebesgue measure on $\mathbb{R}^d$. That means,  we suppose that there exists a zero mean Gaussian random field $\{W_1(t,x), t\ge 0, x\in \mathbb{R}^d\}$  with covariance
 \[
 E(W_1(t,x)W_1(s,y))= (s\wedge t) q(x,y),
 \]
such that  $W_t(\varphi)=      \int_{\mathbb{R}^d} \varphi( x) W_1(t,x)   dx$,  for any $\varphi \in C_0(\mathbb{R}^d)$,  where $q(x,y)$ is positive definite, namely,
 $\int_{\mathbb{R}^{2d}} q(x,y)f(x)f(y) dxdy\ge 0$ for all $f\in L^2(\mathbb{R}^d, dx)$.    The additional regularity conditions imposed on $q$ have allowed us to introduce the density process $W_1(t,x)$, which is a Brownian motion in the time variable and it has the spacial covariance $q$.

From a Theorem of Mercer's type (section 98 on page 245 in \cite{riesy}) we know that
if $\int_{\mathbb{R}^{2d} } |q(x, y)|^2 dxdy<\infty$,  then
\[
q(x, y)=\sum_{n=1}^\infty \lambda_n  e_n(x) e_n(y)\,,\quad
\]
where $e_n, n=1, 2, \cdots$ is an orthonormal basis of $L^2(\mathbb{R}^d, dx)$ and $\sum_{n=1}^\infty
\lambda_n^2<\infty$.  The positive definite property of $q(x,y)$ implies
$\lambda_n\ge 0$.   If we take $C(x, y)=\sum_{n=1}^\infty \sqrt{\lambda_n }  e_n(x) e_n(y)$, then
$q(x, y)=\int_{\mathbb{R}^d}c(\xi,x)c(\xi,y)d\xi$.  Thus it is without loss of generality
for us to assume   that
$q(y_1,y_2)=\int_{\mathbb{R}^d}c(\xi,y_1)c(\xi,y_2)d\xi$   for some
$c(\xi,y)$.    Furthermore, we assume $c(x,y)$  has  polynomial growth.

The following is the main result of this section.

\begin{theorem}\label{smooth}
Assume that $q$ is $\gamma_0$-H\"older continuous for some
$\gamma_0>0$ and satisfies (\textbf{H1a}).  Suppose
\begin{equation}  \label{beta}
q(x_1,x_2)\le C(1+|x_1|^\beta+|x_2|^\beta)
\quad\hbox{for some $\beta \in[0,2)  $}\,.
\end{equation}
 Let  $u_0$ be  bounded and $\rho$-H\"older continuous
for  some $\rho>0$. Suppose that there is a
$x_0\in \mathbb{R}^d $  such that $q(x_0,x_0)>0$ and
 $\sigma(u_0(x_0))\not=0$. Then,
 \begin{itemize}
 \item[(1)] If $b$ and $\sigma$ are  continuous differentiable functions with bounded first order derivatives,  for any $t>0$ and $x\in\mathbb{R}^d$, the probability law of $u(t,x)$ is absolutely continuous with respect to the Lebesgue measure.
 \item[(2)] If  $b$ and $\sigma$ be infinitely differentiable with bounded derivatives of all orders, then for any  $t>0$ and $x\in
\mathbb{R}^d$,
  the probability law of $u(t,x)$  has a smooth density
with respect to Lebesgue measure
\end{itemize}
\end{theorem}

\begin{proof} First we claim that  for all $(t,x)$ the random variable $u(t,x)$ belongs to the Sobolev space $\mathbb{D}^{1,2}$  under condition (1), and to the space $\mathbb{D}^{\infty}$ under condition (2). This follows from standard arguments and we omit the proof (see, for instance \cite{nualart}, Proposition 2.4.4 in the case of the stochastic heat equation).  On the other hand,  the  Malliavin derivative $D_{s,y}u(t,x)$  satisfies the linear stochastic evolution equation
\begin{align*}
D_{s,y}u(t,x)=&p_{t-s}(
x-y)\sigma(u(s,y))+\int_{\mathbb{R}^d}\int_0^t
p_{t-r}( x-z)b'(u(r,z))D_{s,y}u(r,z)dr dz \\
  & + \int_{\mathbb{R}^d}  \int_0^t
p_{t-r}( x-z)\sigma'(u(r,z))D_{s,y}u(r,z)W_1(dr,z)dz\,.
\end{align*}
Denote
\[
F:=\|Du(t,x)\|_{\mathcal{H}_\infty}^2=\displaystyle
\int_0^t\int_{\mathbb{R}^{2d}}D_{s,y_1}u(t,x)D_{s,y_2}u(t,x)q(y_1,y_2)dy_1dy_2ds\,,
\]
where $\|\cdot\|_{\mathcal{H}_\infty}$ is the Hilbert norm
introduced in Section 2. We are going to show only  the statement
(2), and the first one follows from similar arguments. It suffices
to show that    $E\left[ F^{-p}\right] <\infty$ for any $p\ge 1$. We
divide the proof into two   steps.

\noindent\textbf{Step 1}. \quad
Introduce   $V_{s,\xi}(t,x)=\displaystyle\int_{\mathbb{R}^d} c(\xi,y)D_{s,y}u(t,x)dy$.
Then we can write
\[
F=\int_0^t\int_{\mathbb{R}^{d}} V_{s,\xi}(t,x)^2d\xi ds\,.
\]
For any fixed $(s,\xi)$, the random field $\{V_{s,\xi}(t,x), t\ge s, x \in \mathbb{R}^d\}$ satisfies the following
linear stochastic  heat equation for  $t\ge s$, and $ x\in \mathbb{R}^d$,
\begin{align*}
\begin{cases}
\dfrac{\partial V_{s,\xi}}{\partial t}=\dfrac{1}{2}\triangle V_{s,\xi}+b'(u)V_{s,\xi}+\sigma'(u)
V_{s,\xi}\dfrac{\partial W_1}{\partial t}(t,x),  &\\
V_{s,\xi}(s,x)=c(\xi,x)\sigma(u(s,x))\,.   &
\end{cases}
\end{align*}
Consider the continuous semimartingale $\{F(t,x), t\ge 0, x\in \mathbb{R}^d\}$ given by
\[
F(t,x)  = \int_0^t b'(u(r,x))dr+\int_0^t
\sigma'(u(r,x))W_1(dr,x).
\]
The local characteristic of this semimartingale are $b(t,x)=b'(u(t,x))$ and
\[
a(t,x,y) =  \int_0^t\sigma'(u(r,x))\sigma'(u(r,y))q(x,y)dr\,.
\]
Notice that  conditions  (\ref{eq7}) and (\ref{eq8})  hold  because  $b'$ and  $\sigma'$ are bounded and $q$ satisfies  (\ref{beta}).
Then, Theorem  3.1 gives an explicit Feynman-Kac formula for the above  equation. This
means that we have
\begin{align*}
V_{s,\xi}(t,x)=&E^B\bigg[c(\xi,x+B_{t-s})\sigma(u(s,x+B_{t-s}))\\
& \quad \times \exp\left\{\int_s^t
F(dr,x+B_{t-s}-B_{r-s})-\frac{1}{2}\int_s^t
\bar a(r+x,B_{t-s}-B_{r-s})dr\right\}\bigg].
\end{align*}

\noindent
\textbf{Step 2}. \quad  Let
\[
Y(s,t;B)=\displaystyle\int_s^t
F(dr,x+B_{t-s}-B_{r-s})-\frac{1}{2}\int_s^t
\bar a(r+x,B_{t-s}-B_{r-s})dr.
\]
  Then
\begin{align*}
&\int_0^t\int_{\mathbb{R}^d}|V_{s,\xi}(t,x)|^2d\xi ds
=\int_0^t\int_{\mathbb{R}^d}E^{B,\tilde
B}\bigg[c(\xi,x+B_{t-s})c(\xi,x+\tilde
B_{t-s})\\
&\quad  \times  \sigma(u(s,x+B_{t-s}))\sigma(u(s,x+\tilde B_{t-s}))\exp\{Y(s,t;B)+Y(s,t;\tilde B)\}\bigg]d\xi ds\\
 &=\int_0^t E^{B,\tilde B} \bigg[ q(x+B_{t-s},x+\tilde B_{t-s})\\
&\quad \times \sigma(u(s,x+ B_{t-s}))\sigma(u(s,x+\tilde
B_{t-s}))\exp\{Y(s,t;B)+Y(s,t;\tilde B)\} \bigg] ds \\
& =\int_0^t H(s) ds,
\end{align*}
where $\tilde B$ is a standard Brownian motion  independent of $B$.
  If we can show that  $E(H(0)^{-p})<\infty$ for all
$p>1$, and $H(s)$ is H\"older continuous, then by   Lemma
\ref{lemma7} below we deduce
\[
E\left(\int_0^t\int_{\mathbb{R}^d}|V_{s,\xi}(t,x)|^2d\xi
ds\right)^{-p} =E\left(\int_0^t H(s) ds\right)^{-p}
<\infty
\]
for all $p\ge 1$.
The H\"older continuity of $H(s)$
can be verified from  the following inequality:
\[
E|H(s_1)-H(s_2)|^p\le C
|s_2-s_1|^{\frac{p}{2}\min\{\rho,\gamma_0,1+\gamma\}},
\]
 where $C$
is determined by
\[
 \sup\limits_{s\in[0,t]}\bigg\{
E|q(x+B_{t-s},x+\tilde B_{t-s})|^{8p} ,E|\sigma(u(s,x+
B_{t-s}))|^{8p},E\exp\{8pY(s,t;B)\}\bigg\}.
\]
It remains to show that $E(H(0)^{-p})<\infty$. Notice that
\[
H(0)=E^{B,\tilde B} \left( G_x  \exp\{Y(0,t;B)+Y(0,t;\tilde B)\}\right),
\]
where
\[
  G_{x}=q(x+B_{t},x+\tilde B_{t}) \sigma(u_0(x+ B_{t}))\sigma(u_0(x+\tilde B_{t})).
  \]
We can write, by Jensen's inequality,
\begin{align*}
&E\bigg(E^{B,\tilde B}\bigg[G_x\exp\{Y(0,t;B)+Y(0,t;\tilde B)\}\bigg]\bigg)^{-p}\\
=&E\left|E^{B,\tilde B}\bigg[\left|G_x\right|   \mathrm{sign}(G_x)\exp\{Y(0,t;B)+Y(0,t;\tilde B)\}\bigg]\right|^{-p}\\
\le & \bigg[E^{B,\tilde B}
\left|G_x\right|\bigg]^{-p-1}E\bigg[\left|G_x\right|\exp\{-p(Y(0,t;B)+Y(0,t;\tilde
B))\} \bigg]\,.
\end{align*}
Our nondegeneracy hypotheses imply that $E^{B,\tilde B} G>0$, and this allows us to conclude the proof.
\end{proof}

\begin{lemma}\label{lemma7} Let
$\left\{S_t,{0\le t\le 1}\right\}$ be  a non-negative stochastic
process. If $ES_0^{-a}<\infty$ for some $a>0,$ and
$\sup\limits_{0\le s\le t}|S_s-S_0|\le Gt^\gamma$  where G is a
positive random
variable with $EG^b<\infty$ for some $b>0,$ then we have\\
\[
E\left|\int_0^1 S_tdt\right|^{-p}<\infty, \;\;\text{for}\;\;\; 0<p<ab\gamma/(a+b+b\gamma)\,.
\]
In particular, if $a$ and $b$ can be arbitrarily large, then $p$ can
also be chosen arbitrarily large.
\end{lemma}
\begin{proof}
Let  $\alpha,\beta>0,$ where $\alpha+\beta<1$ and
$b\beta\gamma-b\alpha\ge a\alpha$,  and
$0<\epsilon<2^{ \alpha+\beta-1}$.   We have
\begin{align*}
& P\left[\int_0^1 S_t dt < \epsilon \right]
\le   P\left[\int_0^{\epsilon^\beta}S_tdt<\epsilon,\quad
S_0>\epsilon^\alpha\right]+P\left[S_0<\epsilon^\alpha\right]\\
\le & P\left[\sup\limits_{0\le t\le \epsilon^\beta}
|S_t-S_0|>\dfrac{1}{2}\epsilon^\alpha\right]+
P\left[S_0^{-a}>\epsilon^{-a\alpha}\right]\\
\le & 2^b\epsilon^{-b\alpha} E\left(\sup\limits_{0\le t\le
\epsilon^\beta}|S_t-S_0|^b\right)+ \epsilon^{a\alpha}ES_0^{-a}
\le   C\left(
\epsilon^{b\beta\gamma-b\alpha}+\epsilon^{a\alpha}\right)
\le   C\epsilon^{a\alpha}\,.
\end{align*}
Then $E\left|\int_0^1 S_tdt\right|^{-p}<\infty,$ for
$0<p<a\alpha.$\
The lemma follows with the choice of   $\alpha $ and $\beta$ such that
$\alpha<b\gamma/(a+b+b\gamma)$ and $\beta=(a+b)/(a+b+b\gamma)$.
\end{proof}

\medskip
\parindent=0pt
Yaozhong Hu and  David Nualart\\
Department of Mathematics \\
University of Kansas \\
Lawrence, Kansas, 66045 \\
and\\
Jian Song \\
Department of Mathematics\\
Rutgers University\\
Hill Center - Busch Campus\\
110 Frelinghuysen Road\\
Piscataway, NJ 08854-8019

\end{document}